\documentclass{amsart}
\usepackage{amsfonts, amsbsy, amsmath, amssymb}

\ifx\fiverm\undefined 
  \newfont\fiverm{cmr5} 
\fi

\newtheorem{thm}{Theorem}[section]
\newtheorem{lem}[thm]{Lemma}

\newtheorem{rmk}[thm]{Remark}

\newtheorem{thm-con}[thm]{Theorem-Conjecture}
\numberwithin{equation}{section}

\theoremstyle{definition}

\allowdisplaybreaks

\newcommand{\f}{\Bbb F}
\newcommand{\pf}{\Bbb P^1(\f_q)}

\begin{document}

\title[Carlitz formula]{A power sum formula by Carlitz and its applications to permutation rational functions of finite fields}

\author[Xiang-dong Hou]{Xiang-dong Hou}
\address{Department of Mathematics and Statistics,
University of South Florida, Tampa, FL 33620}
\email{xhou@usf.edu}



\begin{abstract}
A formula discovered by L. Carlitz in 1935 finds an interesting application in permutation rational functions of finite fields. It allows us to determine all rational functions of degree three that permute the projective line $\Bbb P^1(\f_q)$ over $\f_q$, a result previously obtained by Ferraguti and  Micheli through a different method. It also allows us to determine all rational functions of degree four that permute $\Bbb P^1(\f_q)$ under a certain condition. (A complete determination of all rational functions of degree four that permute $\Bbb P^1(\f_q)$ without any condition will appear in a separate forthcoming paper.)
\end{abstract}

\maketitle

\section{Introduction}

Let $\f_q$ denote the finite field with $q$ elements. In 1935, while studying what later referred as the Carlitz module, L. Carlitz \cite{Carlitz-DMJ-1935} found a closed formula for the sum $\sum f^{-k}$, where $k$ is a fixed positive integer and $f$ runs through all monic polynomials of degree $d$ in $\f_q[X]$; also see \cite{Hicks-Hou-Mullen-AMM-2012} and \cite[\S 2.3]{Thakur-FFA-2009}. A special case of Carlitz's formula \cite[Theorem~9.4]{Carlitz-DMJ-1935} states that 
\[
\sum_{x\in\f_q}\frac 1{(X+x)^k}=\frac 1{(X-X^q)^k},\quad 1\le k\le q.
\]
(Note: In \cite[Theorem~9.4]{Carlitz-DMJ-1935}, the right side should be $(-1)^{km}/L_k^m$; see \cite[Theorem~2.2]{Hicks-Hou-Mullen-AMM-2012}.) For our purpose, it is more convenient to write the above equation as
\begin{equation}\label{1.1}
\sum_{x\in\f_q}\frac 1{(x-X)^k}=\frac 1{(X^q-X)^k},\quad 1\le k\le q.
\end{equation}
The main purpose of the present paper is to demonstrate a nice application of \eqref{1.1} in permutation rational functions of finite fields.

Let $\Bbb P^1(\f_q)=\f_q\cup\{\infty\}$ be the projective line over $\f_q$. A rational function in $\f_q(X)$ defines a mapping from $\Bbb P^1(\f_q)$ to itself. For $0\ne f(X)=P(X)/Q(X)\in\f_q(X)$, where $P,Q\in\f_q[X]$ and $\text{gcd}(P,Q)=1$, we define $\deg f=\max\{\deg P,\deg Q\}$. Let $G(q)=\{\phi\in\f_q(X):\deg\phi=1\}$. Then $(G(q),\circ)$ is a group which is isomorphic to $\text{PGL}(2,\f_q)$. Elements of $G(q)$ induce permutations of $\Bbb P^1(\f_q)$.  Two rational functions $f,g\in\f_q(X)$ are called {\em equivalent} if there exist $\phi,\psi\in G(q)$ such that $f=\phi\circ g\circ \psi$. A polynomial $f(X)\in\f_q[X]$ that permutes $\f_q$ is called a {\em permutation rational polynomial} (PP) of $\f_q$; a rational function $f(X)\in\f_q(X)$ that permutes $\Bbb P^1(\f_q)$ is called a {\em permutation rational function} (PR) of $\Bbb P^1(\f_q)$. There is an extensive literature on permutation polynomials. In particular, PPs of degree $\le 7$, including degree 8 in characteristic 2, have been classified \cite{Dickson-AM-1896, Fan-arXiv1812.02080, Fan-arXiv1903.10309, Li-Chandler-Xiang-FFA-2010, Shallue-Wanless-FFA-2013}. However, little is known about PRs of low degree of $\pf$. Recently, Ferraguti and Micheli \cite{Ferraguti-Micheli-arXiv1805.03097} classified PRs of degree $3$ of $\pf$ using the Chebotarev theorem for function fields. We will see that \eqref{1.1} allows us to reach the same conclusion fairly quickly. Moreover, \eqref{1.1} allows us to determine all PRs of degree $4$ of $\pf$ under a certain condition. A complete determination of all PRs of degree $4$ of $\pf$ without any condition will appear in a separate forthcoming paper \cite{Hou-ppt}. 

Section~2 describes in general how \eqref{1.1} is used in the study of permutation properties of rational functions. PRs of $\pf$ of degree 3 and degree 4 are discussed in Sections~3 and 4, respectively. Some computations are beyond manual capacity, but all of them are easily handled with computer assistance.

\section{Power Sums of Rational Functions} 

Let $f(X)\in\f_q(X)$. Under equivalence, we may assume that $f(\infty)=\infty$; such a rational function permutes $\pf$ if and only if it permutes $\f_q$. By Hermite's criterion (\cite[Lemma~2.21]{Hou-ams-gsm-2018}, \cite[Lemma~7.3]{Lidl-Niederreiter-FF-1997}), $f(X)$ permutes $\f_q$ if and only if 
\begin{equation}\label{2.1}
\sum_{x\in\f_q}f(x)^s=\begin{cases}
0&\text{if}\ 1\le s\le q-2,\cr
-1&\text{if}\ s=q-1.
\end{cases}
\end{equation}
Let the partial fraction decomposition of $f(X)^s$ be
\begin{equation}\label{2.2}
f(X)^s=\sum_{i=0}^ma_iX^i+\sum_{i=1}^n\frac{b_i}{(X-r_i)^{k_i}},
\end{equation}
where $a_0,\dots,a_m\in\f_q$, $b_1,\dots,b_n\in\overline \f_q$ (the algebraic closure of  $\f_q$), $r_1,\dots,r_n,\in\overline\f_q\setminus\f_q$, and $k_1,\dots,k_n$ are positive integers. Then
\begin{equation}\label{2.3}
\sum_{x\in\f_q}f(x)^s=-\sum_{\substack{0<i\le m\cr i\equiv 0\,(\text{mod}\,q-1)}}a_i+\sum_{i=1}^n b_i\sum_{x\in\f_q}\frac 1{(x-r_i)^{k_i}}.
\end{equation}
If $m<q-1$ and $k_i\le q$ for all $1\le i\le n$, then by \eqref{1.1} we have
\begin{equation}\label{2.4}
\sum_{x\in\f_q}f(x)^s=\sum_{i=1}^n \frac{b_i}{(r_i^q-r_i)^{k_i}}.
\end{equation}
Our approach is based on \eqref{2.4}.

\section{PRs of Degree $3$}

A polynomial $f(X)\in\f_q[X]$ permutes $\pf$ if and only if $f(X)$ is a PP of $\f_q$. PPs of low degree of $\f_q$ are known. For example, a cubic PP of $\f_q$ is equivalent to (i) $X^3$ where $q\not\equiv 1\pmod 3$, or (ii) $X^3-aX$ where $\text{char}\,\f_q=3$ and $a$ is a nonsquare in $\f_q$; a quartic PP of $\f_q$ is equivalent to (i) $X^4+3X$ where $q=7$, or (ii) $X^4+aX^2+bX\in\f_q[X]$ where $q$ is even and $X^4+aX^2+bX$ has no nonzero root in $\f_q$; see \cite[Table~7.1]{Lidl-Niederreiter-FF-1997}.

When considering PRs of $\pf$ of low degree, we are mainly interested in those that are not equivalent to polynomials. It is easy to see that if $f(X)$ is a PR of $\pf$ of degree $d$ which is not equivalent to a polynomial, then $f(X)$  is equivalent to $P(X)/Q(X)$, where $P,Q\in\f_q[X]$, $\text{gcd}(P,Q)=1$, $d=\deg P>\deg Q$ and $Q$ has no root in $\f_q$.

Let $f(X)$ be a PR of $\pf$ of degree $3$ which is not equivalent to a polynomial. Then $f(X)$ equivalent to
\begin{equation}\label{3.2}
X+\frac b{X-r}+\frac{b^q}{X-r^q},
\end{equation}
where $r\in\f_{q^2}\setminus\f_q$ and $b\in\f_{q^2}^*$. If $\text{char}\,\f_q\ne 2$, we may further assume that $r+r^q=0$.

\begin{lem}\label{L3.2}
Let $f$ denote the rational function in \eqref{3.2}. If $f$ is a PR of $\Bbb P^1(\f_q)$, then $b\in\f_q^*$.
\end{lem}

\begin{proof}
Since $f$ fixes $\infty$, it permutes $\f_q$. Let $r_1=r$, $r_2=r^q$, $b_1=b$, $b_2=b^q$. If $q>2$, by \eqref{2.1} and \eqref{2.4}, we have
\[
0=\sum_{x\in\f_q}f(x)=\frac{b_1}{r_2-r_1}+\frac{b_2}{r_1-r_2}=-\frac{b_1-b_2}{r_1-r_2}.
\] 
If $q=2$, by \eqref{2.1} and \eqref{2.3}, the above equation is modified as
\[
-1=\sum_{x\in\f_q}f(x)=-1+\frac{b_1}{r_2-r_1}+\frac{b_2}{r_1-r_2}=-1-\frac{b_1-b_2}{r_1-r_2}.
\] 
Hence we always have $b_1=b_2$, i.e., $b\in\f_q^*$.
\end{proof}

Now assume that $b\in\f_q^*$ in \eqref{3.2}. Then the rational function \eqref{3.2} is equivalent to
\begin{equation}\label{3.3}
aX+\frac 1{X-r}+\frac 1{X-r^q},
\end{equation}
where $a\in\f_q^*$. By a substitution $X\mapsto uX$, $u\in\f_q^*$, we may replace $a$ with $au^2$. Therefore, when $q$ is even, we may assume that $a=1$; when $q$ is odd, we may assume that $a$ is either $1$ or a chosen nonsquare in $\f_q^*$. However, when $q$ is odd, we will not take advantage of the normalization of $a$ since there is not much to be gained.

\begin{thm}\label{T3.3}
Let $q$ be even and 
\[
f(X)=X+\frac 1{X-r}+\frac 1{X-r^q}\in\f_q(X),
\]
where $r\in\f_{q^2}\setminus\f_q$. Then $f$ is a PR of $\Bbb P^1(\f_q)$ if and only if $r+r^q=1$.
\end{thm}

\begin{proof}
($\Rightarrow$) We know that $f$ permutes $\f_q$. Let $r_1=r$ and $r_2=r^q$. We have
\begin{align*}
f(X)^3=\,&X^3+r_1+r_2+\frac{1+r_1^2}{X-r_1}+\frac{1+r_2^2}{X-r_2}+\frac{1+r_1r_2+r_1^2}{r_1+r_2}\frac 1{(X-r_1)^2}\cr
&+\frac{1+r_1r_2+r_2^2}{r_1+r_2}\frac 1{(X-r_2)^2}+\frac 1{(X-r_1)^3}+\frac 1{(X-r_2)^3}.
\end{align*}
By \eqref{2.4}, for $q>4$,
\begin{align*}
\sum_{x\in\f_q}f(x)^3=\,&(1+r_1^2)\frac1{r_2-r_1}+(1+r_2^2)\frac 1{r_1-r_2}+\frac{1+r_1r_2+r_1^2}{r_1+r_2}\frac 1{(r_2-r_1)^2}\cr
&+\frac{1+r_1r_2+r_2^2}{r_1+r_2}\frac 1{(r_1-r_2)^2}+\frac 1{(r_2-r_1)^3}+\frac 1{(r_1-r_2)^3}\cr
=\,&\frac{(1+r_1+r_2)^2}{r_1+r_2}.
\end{align*}
It follows from \eqref{2.1} that $r_1+r_2=1$. If $q=4$, the above computation becomes
\[
\sum_{x\in\f_q}f(x)^3=1+\frac{(1+r_1+r_2)^2}{r_1+r_2},
\]
and by \eqref{2.1}, we also have $r_1+r_2=1$. If $q=2$, obviously $r_1+r_2=1$.

\medskip
($\Leftarrow$) It suffices to show that $f$ permutes $\f_q$. Assume to the contrary that there exist $x\in\f_q$ and $y\in\f_q^*$ such that $f(x+y)-f(x)=0$.
We have
\begin{align*}
&f(x+y)-f(x)=\cr
&\frac{y\bigl[(1+r+r^2+x+x^2)^2+(1+r+r^2+x+x^2)y+(r+r^2+x+x^2)y^2 \bigr]}{(x+r)(x+r+1)(x+y+r)(x+y+r+1)}.
\end{align*}
Thus
\[
(1+r+r^2+x+x^2)^2+(1+r+r^2+x+x^2)y+(r+r^2+x+x^2)y^2 =0,
\]
i.e.,
\[
r+r^2+x+x^2=\frac{1+r+r^2+x+x^2}{y}+\Bigl(\frac{1+r+r^2+x+x^2}{y} \Bigr)^2.
\]
It follows that
\[
0=\text{Tr}_{q/2}(r+r^2+x+x^2)=\text{Tr}_{q/2}(r+r^2).
\]
However, since $X^2+X+(r+r^2)\in \f_q[X]$ has a root $r\notin\f_q$, it is irreducible over $\f_q$. Thus $\text{Tr}_{q/2}(r+r^2)=1$, which is a contradiction. 
\end{proof}

\begin{rmk}\label{R3.3}\rm
In fact, all PRs in Theorem~\ref{T3.3} are equivalent. Let $f_r(X)$ denote the PR in Theorem~\ref{T3.3}, where $r+r^q=1$. If $r_1,r_2\in\f_{q^2}$ are such that $r_1+r_1^q=1=r_2+r_2^q$, then $u:=r_1-r_2\in\f_q$. We have
\begin{align*}
f_{r_1}(X+u)\,&=X+u+\frac 1{X+u-r_1}+\frac 1{X+u-r_1^q}=X+u+\frac 1{X-r_2}+\frac 1{X-r_2^q}\cr
&=f_{r_2}(X)+u.
\end{align*}
\end{rmk}

\begin{thm}\label{T3.4}
Let $q$ be odd and 
\[
f(X)=aX+\frac 1{X-r}+\frac 1{X+r}\in\f_q(X),
\]
where $a\in\f_q^*$, $r\in\f_{q^2}\setminus\f_q$, $r^2\in\f_q$. Then $f$ is a PR of $\Bbb P^1(\f_q)$ if and only if $a=-1/4r^2$.
\end{thm}

\begin{proof}
($\Rightarrow$) We have
\[
f(X)^2=a^2X^2+4a+\frac{2ar^2+1}r\frac 1{X-r}-\frac{2ar^2+1}r\frac 1{X+r}+\frac 1{(X-r)^2}+\frac 1{(X+r)^2}.
\]
By \eqref{2.4}, for $q>3$, 
\begin{align*}
\sum_{x\in\f_q}f(x)^2=\,&\frac{2ar^2+1}r \frac1{-r-r}-\frac{2ar^2+1}r \frac1{r-(-r)}+\frac 1{(-r-r)^2}+\frac 1{(r-(-r))^2}\cr
=\,&-\frac{1+4ar^2}{2r^2}.
\end{align*}
Therefore by \eqref{2.1}, $a=-1/4r^2$. If $q=3$, we have 
\[
\sum_{x\in\f_q}f(x)^2=-a^2-\frac{1+4ar^2}{2r^2}=-1-\frac{1+4ar^2}{2r^2},
\]
and by \eqref{2.1}, we also have $a=-1/4r^2$.

\medskip
($\Leftarrow$) Assume to the contrary that there exist $x\in\f_q$ and $y\in\f_q^*$ such that $f(x+y)-f(x)=0$. We have
\[
f(x+y)-f(x)=-\frac{y(3r^2+x^2+xy-ry)(3r^2+x^2+xy+ry)}{4r^2(x-r)(x+r)(x+y-r)(x+y+r)}.
\]
It follows that $3r^2+x^2+xy=\pm ry$. Since $r\in\f_{q^2}\setminus\f_q$ and $r^2\in\f_q$, we must have $y=0$, which is a contradiction.
\end{proof}

\begin{rmk}\label{R3.5}\rm
All PRs in Theorem~\ref{T3.4} are equivalent. Let $f_r(X)$ denote the PR in Theorem~\ref{T3.4}, where $r^2$ is a nonsquare of $\f_q$ and $a=-1/4r^2$. If $r_1,r_2\in\f_{q^2}$ are such that $r_1^2$ and $r_2^2$ are nonsquares of $\f_q$, then $u:=r_1/r_2\in\f_q^*$ and $f_{r_1}(uX)=u^{-1}f_{r_2}(X)$. When $\text{char}\,\f_q=3$, $f_r(X)=-X^3/(r^2(X^2-r^2))$, which is actually equivalent to the polynomial $r^2X^3-X$.
\end{rmk}


\section{PRs of Degree $4$}

Let $f$ be a PR of $\Bbb P^1(\f_q)$ of degree $4$ which is not equivalent to a polynomial. Then $f$ is either equivalent to 
\begin{equation}\label{3.4}
\frac{P(X)}{(X-r)(X-r^q)},
\end{equation}
where $P(X)\in\f_q[X]$, $\deg P=4$, $r\in\f_{q^2}\setminus\f_q$ and $P(r)\ne 0$, or equivalent to
\begin{equation}\label{3.5}
\frac{Q(X)}{(X-r)(X-r^q)(X-r^{q^2})},
\end{equation}
where $Q(X)\in\f_q[X]$, $\deg Q=4$, $r\in\f_{q^3}\setminus\f_q$ and $Q(r)\ne 0$. Note that \eqref{3.4} is further equivalent to 
\begin{equation}\label{3.6}
aX^2+bX+\frac c{X-r}+\frac{c^q}{X-r^q},
\end{equation}
where $a\in\f_q^*$, $b\in\f_q$, $c\in\f_{q^2}^*$, and \eqref{3.5} is further equivalent to 
\begin{equation}\label{3.7}
aX+\frac b{X-r}+\frac{b^q}{X-r^q}+\frac{b^{q^2}}{X-r^{q^2}},
\end{equation}
where $a\in\f_q^*$ and $b\in\f_{q^3}^*$. 

Let $f$ be of the form \eqref{3.7} or \eqref{3.8}. We will compute the power sums $\sum_{x\in\f_q}f(x)^k$ for various values of $k$ (in one case, for $k$ up to $11$). The method has already been illustrated in the proofs of Theorems~\ref{T3.3} and \ref{T3.4}: first find the partial fraction decomposition of $f(X)^k$ and then use formula \eqref{2.4}. The intermediate steps of the computations are rather involved and they are usually omitted. Readers should have no difficulty to verify the results with computer assistance.

\subsection{PRs of the form \eqref{3.6}}\

\medskip

Let $f$ be a PR of $\pf$ of the form \eqref{3.6}. We assume that $q>3$. It follows from $\sum_{x\in\f_q}f(x)=0$ that $c\in\f_q^*$; see the proof of Lemma~\ref{L3.2}. Hence we may assume that $c=1$. If $q$ is odd, we may further assume that $r+r^q=0$. If $q$ is even, by a suitable substitution $X\mapsto uX$, $u\in\f_q^*$, we may assume that $b=0$ or $1$.

\begin{thm}\label{T3.5}
Let $q=2^n$, $n\ge 5$, and let 
\[
f(X)=aX^2+bX+\frac 1{X-r}+\frac 1{X-r^q},
\]
where $a\in\f_q^*$, $b\in\f_2$, $r\in\f_{q^2}\setminus\f_q$. Then $f$ is not a PR of $\Bbb P^1(\f_q)$.
\end{thm}

\begin{proof} 
Assume to the contrary that $f$ is a PR of $\Bbb P^1(\f_q)$. Let $r_1=r$ and $r_2=r^q$.

\medskip
{\bf Case 1.} Assume that $b=1$. We have
\begin{align*}
&f(X)^3=\cr
&a^3X^6+a^2X^5+aX^4+X^3+a^2(r_1+r_2)X^2+a^2(r_1+r_2)^2X+r_1+r_2\cr
&+a^2(r_1^3+r_2^3)+\frac {1+r_1^2+a^2r_1^4}{X-r_1}+\frac {1+r_2^2+a^2r_2^4}{X-r_2}+\frac{1+r_1r_2+r_1^2+ar_1^2r_2+ar_1^3}{(r_1+r_2)(X-r_1)^2}\cr
&+\frac{1+r_1r_2+r_2^2+ar_1r_2^2+ar_2^3}{(r_1+r_2)(X-r_2)^2}+\frac1{(X-r_1)^3}+\frac1{(X-r_2)^3}.
\end{align*}
By \eqref{2.4}, we find that 
\begin{align*}
\sum_{x\in\f_q}f(x)^3=\,&\frac {1+r_1^2+a^2r_1^4}{r_2-r_1}+\frac {1+r_2^2+a^2r_2^4}{r_1-r_2}+\frac{1+r_1r_2+r_1^2+ar_1^2r_2+ar_1^3}{(r_1+r_2)(r_2-r_1)^2}\cr
&+\frac{1+r_1r_2+r_2^2+ar_1r_2^2+ar_2^3}{(r_1+r_2)(r_1-r_2)^2}+\frac1{(r_2-r_1)^3}+\frac1{(r_1-r_2)^3}\cr
=\,&\frac 1{r_1+r_2}\bigl[1+a(r_1+r_2)\bigr]\bigl[1+(r_1+r_2)^2+a(r_1+r_2)^3\bigr].
\end{align*}
Hence
\begin{equation}\label{3.8}
a=(r_1+r_2)^{-1}\quad \text{or}\quad (r_1+r_2)^{-1}+(r_1+r_2)^{-3}.
\end{equation}
In the same way, we find that
\[
\sum_{x\in\f_q}f(x)^7=\frac{(1+a(r_1+r_2))h}{(r_1+r_2)^3},
\]
where 
\begin{align*}
h=\,& a^5 (r_1^{13}+r_1^{12}
r_2+r_1 r_2^{12}+r_2^{13})+a^4
(r_1^{12}+r_1^{10}+r_2^{12}+r_2^{10})\cr
&+a^3
(r_1^9 r_2^2+r_1^9+r_1^8 r_2^3+r_1^7
r_2^4+r_1^7 r_2^2+r_1^6 r_2^5+r_1^6
r_2^3+r_1^5 r_2^6+r_1^5 r_2^4+r_1^4
r_2^7\cr
&+r_1^4 r_2^5+r_1^3 r_2^8+r_1^3
r_2^6+r_1^2 r_2^9+r_1^2 r_2^7+r_2^9)\cr
&+a^2
(r_1^8 r_2^2+r_1^8+r_1^7 r_2+r_1^6
r_2^4+r_1^5 r_2^3+r_1^4 r_2^6+r_1^3
r_2^5+r_1^2 r_2^8+r_1^2 r_2^2+r_1
r_2^7+r_2^8)\cr
&+a (r_1^9+r_1^8
r_2+r_1^7 r_2^2+r_1^7+r_1^6 r_2^3+r_1^4
r_2^3+r_1^3 r_2^6\cr
&+r_1^3 r_2^4+r_1^2
r_2^7+r_1^2 r_2+r_1 r_2^8+r_1
r_2^2+r_2^9+r_2^7)\cr
&+r_1^8+r_1^6
r_2^2+r_1^5 r_2+r_1^4 r_2^2+r_1^2
r_2^6+r_1^2 r_2^4+r_1^2+r_1
r_2^5+r_1 r_2+r_2^8+r_2^2
\end{align*}
If in \eqref{3.8}, $a=(r_1+r_2)^{-1}+(r_1+r_2)^{-3}$, we find that 
\[
h=(1+r_1+r_2)^2.
\]
Then $r_1+r_2=1$ and hence $a=0$, which is a contradiction. Thus we have $a=(r_1+r_2)^{-1}$. However, this creates another contradiction: $f(r_1+r_2)=f(0)$.

\medskip

{\bf Note.} The sum $\sum_{x\in\f_q}f(x)^5$ is not useful since our computation shows that it is a multiple of $\sum_{x\in\f_q}f(x)^3$. We find that 
\begin{align*}
\sum_{x\in\f_q}f(x)^5=\,&\frac 1{(r_1+r_2)^3}\bigl[1+a(r_1+r_2)\bigr]\bigl[1+(r_1+r_2)^2+a(r_1+r_2)^3\bigr]\cr
&\cdot\bigl[1+(r_1+r_2)^2+(r_1+r_2)^4+a(r_1+r_2)^3+a^2(r_1+r_2)^6\bigr].
\end{align*}

\medskip
{\bf Case 2.} Assume that $b=0$. The proof is similar to that of Case 1. We have
\begin{equation}\label{3.9}
\sum_{x\in\f_q}f(x)^3=a\bigl[1+a(r_1+r_2)^3\bigr].
\end{equation}
Hence $a=(r_1+r_2)^{-3}$. Then $f(X)$ is equivalent to 
\[
X^2+\frac 1{X-r_1'}+\frac 1{X-r_2'},
\]
where $r_1'=r_1/(r_1+r_2)$ and $r_2'=r_2/(r_1+r_2)$. Thus, writing $r_1'$ and $r_2'$ as $r_1$ and $r_2$, we may assume that 
\begin{equation}\label{f}
f(X)=X^2+\frac1{X-r_1}+\frac1{X-r_2},
\end{equation}
where $r_1+r_2=1$. Under this assumption, we find that $\sum_{x\in\f_q}f(x)^k=0$ for $1\le k\le 10$. One might think that $\sum_{x\in\f_q}f(x)^k=0$ for all $1\le k\le q-2$. Surprisingly, however, the pattern stops at $k=11$. We find that 
\[
\sum_{x\in\f_q}f(x)^{11}=1,
\]
which is a contradiction.
\end{proof}

\noindent{\bf Remark.} There is another way to see that the rational function in \eqref{f} cannot be a PR of $\Bbb P^1(\f_q)$ when $q$ is not too small. Write
\[
X^2+\frac1{X-r_1}+\frac1{X-r_2}=X^2+\frac 1{X^2+X+c}=:g(X),
\]
where $c=r_1r_2$ and $\text{Tr}_{q/2}(c)=1$. Assume to the contrary that $g(X)$ is a PR of $\Bbb P^1(\f_q)$. Then for each $x\in\f_q$, the equation $g(x+y)-g(x)=0$ has no solution $y\in\f_q^*$. We have
\[
g(x+y)-g(x)=\frac{y[1+(1+c^2+x^2+x^4)y+(c+x+x^2)y^2+(c+x+x^2)y^3]}{(x^2+x+c)((x+y)^2+(x+y)+c)}.
\]
Hence the equation
\begin{equation}\label{eq-y}
1+(1+c^2+x^2+x^4)y+(c+x+x^2)y^2+(c+x+x^2)y^3=0
\end{equation}
has no solution $y\in\f_q$. The substitution $y\mapsto y+1$ transforms \eqref{eq-y} into
\begin{equation}\label{eq-y1}
y^3+\frac{1+c+c^2+x+x^4}{c+x+x^2}y+c+x+x^2=0,
\end{equation}
which still has no solution in $\f_q$. By \cite[Theorem~1]{Williams-JNT-1975}, we have 
\[
\text{Tr}_{q/2}\Bigl(1+\frac{(1+c+c^2+x+x^4)^3}{(c+x+x^2)^5}\Bigr)=0\quad\text{for all}\ x\in\f_q.
\]
Then by \cite[Theorem~3.1]{Hou-Iezzi-arXiv1906.09487}, when $q$ is not too small,
\begin{equation}\label{P}
1+\frac{(1+c+c^2+X+X^4)^3}{(c+X+X^2)^5}=P(X)+P(X)^2
\end{equation}
for some $P(X)\in\f_q(X)$. However, this is impossible since the valuation of the left side of \eqref{P} at the irreducible polynomial $c+X+X^2$ is $-5$ while the same valuation of the right side of \eqref{P} is either nonnegative or even.

\begin{thm}\label{T3.6}
Assume that $\text{\rm char}\,\f_q\ne 2,3$ and $q>7$ and let 
\[
f(X)=aX^2+bX+\frac 1{X-r}+\frac 1{X+r},
\]
where $a\in\f_q^*$, $b\in\f_q$, $r\in\f_{q^2}\setminus\f_q$, $r^2\in\f_q$. Then $f$ is not a PR of $\Bbb P^1(\f_q)$.
\end{thm}

\begin{proof}
Assume to the contrary that $f(X)$ is a PR of $\Bbb P^1(\f_q)$. Let $r_1=r$ and $r_2=r^q$. We have 
\begin{align*}
f(X)^2=\,&2(2 b + a r_1 + a r_2) + 4 a X + b^2 X^2 + 2 a b X^3 + a^2 X^4\cr
& +
\frac{2 (1 + b r_1^2 + a r_1^3 - b r_1 r_2 - a r_1^2 r_2)}{(r_1 - r_2) (X-r_1)}+
\frac{2 (1 + b r_2^2 + a r_2^3 - b r_1 r_2 - a r_2^2 r_1)}{(r_2 - r_1) (X-r_2)}\cr
&+\frac 1{(X-r_1)^2}+\frac 1{(X-r_2)^2}.
\end{align*}
By \eqref{2.4},
\begin{align}\label{3.10}
\sum_{x\in\f_q}f(x)^2=\,& \frac{2 (1 + b r_1^2 + a r_1^3 - b r_1 r_2 - a r_1^2 r_2)}{(r_1 - r_2)(r_2-r_1)}\\
&+
\frac{2 (1 + b r_2^2 + a r_2^3 - b r_1 r_2 - a r_2^2 r_1)}{(r_2 - r_1)(r_1-r_2)}\cr
&+\frac1{(r_2-r_1)^2} +\frac1{(r_1-r_2)^2}\cr
=\,&-\frac{1+4br^2}{2r^2}.\nonumber
\end{align}
Thus $b=-1/4r^2$. Under this condition, we find that
\[
\sum_{x\in\f_q}f(x)^3=-6a\ne 0,
\]
which is a contradiction.
\end{proof}

\begin{thm}\label{T3.7}
Let $q=3^n$, $n\ge 3$, and let
\[
f(X)=aX^2+bX+\frac 1{X-r}+\frac 1{X+r},
\]
where $a\in\f_q^*$, $b\in\f_q$, $r\in\f_{q^2}\setminus\f_q$, $r^2\in\f_q$. Then $f$ is not a PR of $\Bbb P^1(\f_q)$.
\end{thm}

\begin{proof}
Equation~\eqref{3.10} still holds, hence we have $b=-1/r^2$. Under this condition, We find that
\[
\sum_{x\in\f_q}f(x)^4=0,
\]
but
\[
\sum_{x\in\f_q}f(x)^5=-\frac a{r^2}\ne 0.
\]
Thus we have a contradiction.
\end{proof}

Theorems~\ref{T3.5} -- \ref{T3.7} do not cover PRs of the form \eqref{3.6} for the following small values of $q$: $2,2^2,2^3,2^4,3,3^2,5,7$. For these values of $q$, PRs $f$ of the form \eqref{3.6} are determined by computer search. 

\begin{itemize}
\item $q=2$. $f$ is equivalent to precisely one of the following:
\[
X^2+\frac 1{X^2+X+1},\quad X^2+X+\frac X{X^2+X+1}.
\]

\item $q=2^2$. Fix $u\in\f_{2^2}$ such that $u^2+u+1=0$. Then $f$ is equivalent to precisely one of the following:
\begin{align*}
&X^2+\frac 1{X^2+X+u},\quad uX^2+X+\frac 1{X^2+X+u},\cr
& X^2+X+\frac u{X^2+uX+1},\quad X^2+X+\frac {1+u}{X^2+(1+u)X+1},\cr
&(1+u)X^2+X+\frac 1{X^2+X+u}.
\end{align*}

\item $q=2^3$. Fix $u\in\f_{2^3}$ such that $u^3+u+1=0$. Then $f$ is equivalent to precisely one of the following:
\begin{align*}
&(1+u^2)X^2+X+\frac {u^2}{X^2+u^2X+1},\quad (1+u)X^2+X+\frac u{X^2+uX+1},\cr
& (1+u+u^2)X^2+X+\frac {u+u^2}{X^2+(u+u^2)X+1}.
\end{align*}

\item $q=2^4$. $f$ does not exist.

\item $q=3$. $f$ is equivalent to precisely one of the following:
\[
X^2+\frac {X-1}{X^2+1},\quad X^2+X-\frac 1{X^2+1},\quad X^2+X-\frac {X+1}{X^2+1}.
\]

\item $q=3^2$. $f$ does not exist.

\item $q=5$. $f$ is equivalent to precisely one of the following:
\[
X^2+\frac {2X}{X^2-2},\quad X^2-X+\frac {2X}{X^2-3}.
\]

\item $q=7$. $f$ is equivalent to 
\[
X^2+X+\frac {2X}{X^2-5}.
\]

\end{itemize}

\subsection{PRs of the form \eqref{3.7}}\

\medskip

If $f(X)$ is of the form \eqref{3.7} where $b\in\f_{q^3}^*$ is arbitrary, the power sums $\sum_{x\in\f_q}f(x)^k$ are very complicated and the equations $\sum_{x\in\f_q}f(x)^k=0$ do not seem to produce neat conditions on the parameters $a,b,r$. To simplify the computations, {\em we assume that  $b\in\f_q^*$ in \eqref{3.7}}. (PRs of of the form \eqref{3.7} without this assumption have been determined using a different method and the result will appear in a separate forthcoming paper \cite{Hou-ppt}.) 

Since $b\in\f_q^*$, we may further assume that $b=1$ using equivalence. Therefore, throughout this subsection, 
\begin{equation}\label{3.12}
f(X)=aX+\frac 1{X-r}+\frac 1{X-r^q}+\frac 1{X-r^{q^2}},
\end{equation}
where $a\in\f_q^*$ and $r\in\f_{q^3}\setminus\f_q$. Let $r_1=r$, $r_2=r^q$, $r_3=r^{q^2}$.

\medskip
First assume that $\text{char}\,\f_3\ne 3$ and $q\ge 5$. Then we may further assume that $r_1+r_2+r_3=0$. We find that
\begin{align}
\sum_{x\in\f_q}f(x)\,&=\frac 3{(r_1-r_2)(r_2-r_3)(r_3-r_1)}h_1,\cr
\label{3.13}
\sum_{x\in\f_q}f(x)^2\,&=\frac {-3}{(r_1-r_2)^2(r_2-r_3)^2(r_3-r_1)^2}h_2,\\
\sum_{x\in\f_q}f(x)^3\,&=\frac 3{(r_1-r_2)^3(r_2-r_3)^3(r_3-r_1)^3}h_3, \nonumber
\end{align}
where
\begin{align*}
h_1=\,& r_1^2+r_1 r_2+r_2^2, \cr
h_2=\,& 4 a r_1^6+18 a r_1^5 r_2+12 a r_1^4
r_2^2-26 a r_1^3 r_2^3-18 a r_1^2 r_2^4+6
a r_1 r_2^5\cr
&+4 a r_2^6+3 r_1^4+6 r_1^3
r_2+9 r_1^2 r_2^2+6 r_1 r_2^3+3 r_2^4,\cr
h_3=\,& 4 a^2 r_1^{10}+20 a^2 r_1^9 r_2+33 a^2 r_1^8 r_2^2+12
a^2 r_1^7 r_2^3-36 a^2 r_1^6 r_2^4-66 a^2 r_1^5
r_2^5-36 a^2 r_1^4 r_2^6\cr
&+12 a^2 r_1^3 r_2^7+33 a^2
r_1^2 r_2^8+20 a^2 r_1 r_2^9+4 a^2 r_2^{10}+18 a
r_1^8+99 a r_1^7 r_2+153 a r_1^6 r_2^2\cr
&+18 a r_1^5
r_2^3-144 a r_1^4 r_2^4-171 a r_1^3 r_2^5-36
a r_1^2 r_2^6+45 a r_1 r_2^7+18 a r_2^8-16
r_1^6\cr
&-21 r_1^5 r_2-42 r_1^4 r_2^2-139 r_1^3
r_2^3-177 r_1^2 r_2^4-75 r_1 r_2^5-16 r_2^6.
\end{align*}
We further find that
\begin{align*}
\text{Res}(h_1,h_2;r_1)\,&=2^2\, 3^5 r_2^{12},\cr
\text{Res}(h_1,h_3;r_1)\,&=3^6\, 7r_2^{12},
\end{align*}
where $\text{Res}(h_1,h_2;r_1)$ denotes the resultants of $h_1$ and $h_2$ as polynomials in $r_1$. Note that $\text{Res}(h_1,h_2;r_1)\ne 0$ if $\text{char}\,\f_q\ne 2,3$ and $\text{Res}(h_1,h_3;r_1)\ne 0$ if $\text{char}\,\f_q\ne 3,7$. Therefore we have the following theorem.

\begin{thm}\label{T3.8}
Assume that $\text{\rm char}\,\f_q\ne 3$ and $q\ge 5$. Then $f(X)$ in \eqref{3.12} is not a PR of $\Bbb P^1(\f_q)$.
\end{thm}

\medskip
Next, assume that $\text{\rm char}\,\f_q=3$.

\begin{thm}\label{T3.9}
Assume that $q=3^n$, where $n\ge 3$. Then $f(X)$ in \eqref{3.12} is a PR of $\Bbb P^1(\f_q)$ if and only if $r_1+r_2+r_3=0$ and $a=-(r_1-r_2)^{-2}$. 
\end{thm}

\begin{proof}
Assume that $f(X)$ is a PR of $\Bbb P^1(\f_q)$. We first find that
\[
\sum_{x\in\f_q}f(x)=\frac{(r_1+r_2+r_3)^2}{(r_1-r_2)(r_2-r_3)(r_3-r_1)}.
\]
It follows that 
\begin{equation}\label{3.14}
r_1+r_2+r_3=0.
\end{equation}
Under this condition, $\sum_{x\in\f_q}f(x)^2$ is given by \eqref{3.13}, which is not useful in characteristic $3$. Using \eqref{3.14}, we find that
\[
\sum_{x\in\f_q}f(x)^4=0,
\]
but
\[
\sum_{x\in\f_q}f(x)^5=\frac{a^2}{r_1-r_2}(1-a^2(r_1-r_2)^4).
\]
Thus
\[
a=\epsilon(r_1-r_2)^{-2},\quad \epsilon=\pm1.
\]
Then it is easy to see that $f(X)$ is equivalent to
\[
\epsilon X+\frac1{X-r_1'}+\frac1{X-r_2'}+\frac1{X-r_3'},
\]
where $r_i'=r_i/(r_1-r_2)$. Writing $r_i'$ as $r_i$, we may assume that
\[
f(X)=\epsilon X+\frac1{X-r_1}+\frac1{X-r_2}+\frac1{X-r_3},
\]
where $r_1-r_2=1$, i.e., $r_2=r_1-1$ and $r_3=r_1+1$.

First assume that $\epsilon=1$. We find that 
\[
\sum_{x\in\f_q}f(x)^7=0,
\]
but
\[
\sum_{x\in\f_q}f(x)^8=1,
\]
which is a contradiction.

Now assume that $\epsilon=-1$. In this case, $f(X)=-X-(X^3-X+b)^{-1}$, where $(X-r_1)(X-r_2)(X-r_3)=X^3-X+b$ and $b\in\f_q$ is such that $\text{Tr}_{q/3}(b)\ne 0$. We show that $f(X)$ is a PR of $\Bbb P^1(\f_q)$. Using equivalence, we may write
\[
f(X)=X+\frac 1{X^3-X+b}.
\]
We show that $f(X)$ is a PR of $\pf$. Assume to the contrary that $f(X)$ is not a PR of $\pf$. Then there exist $x\in\f_q$ and $y\in\f_q^*$ such that $f(x+y)-f(x)=0$. Let $c=x^3-x+b$. We have
\[
f(x+y)-f(x)=\frac{y(1+c^2-cy-y^2+cy^3)}{(x^3-x+b)((x+y)^3-(x+y)+b)}.
\]
Hence the equation
\begin{equation}\label{eq-y-0}
1+c^2-cy-y^2+cy^3=0
\end{equation}
has a solution $y\in\f_q$. We claim that $1+c^2\ne 0$. Otherwise, $c\in\f_{3^2}\subset\f_q$ and $\text{Tr}_{3^2/3}(c)=0$. However, $\text{Tr}_{q/3}(c)=\text{Tr}_{q/3}(x^3-x+b)=\text{Tr}_{q/3}(b)\ne0$, which is a contradiction.

The substitution $y\mapsto y+c$ transforms \eqref{eq-y-0} into
\[
1-\alpha^2y^2+\beta y^3=0,
\]
where $\alpha=1/(1 + c^2)$ and $\beta=c/(1 + c^2)^2$. So the equation
\begin{equation}\label{alpha}
y^3-\alpha^2y+\beta=0
\end{equation}
has a solution $y\in\f_q$. However, since 
\[
\text{Tr}_{q/3}\Bigl(\frac\beta{\alpha^3}\Bigr)=\text{Tr}_{q/3}(c(1+c^2))=\text{Tr}_{q/3}(c+c^3)=-\text{Tr}_{q/3}(c)=-\text{Tr}_{q/3}(b)\ne0, 
\]
by \cite[Theorem~2]{Williams-JNT-1975}, \eqref{alpha} has no solution in $\f_q$, which is a contradiction. This completes the proof of the theorem.
\end{proof}

\begin{rmk}\label{R4.6}\rm
The PRs in Theorem~\ref{T3.9} are of the form $X+(X^3-X+b)^{-1}$, where $b\in\f_q$ is such that $\text{Tr}_{q/3}(b)\ne 0$. Clearly, all such rational functions are equivalent. In 2008, Yuan et al. \cite{Yuan-Ding-Wang-Pieprzyk-FFA-2008} proved that for $p=2,3$, if $\delta\in\f_{p^n}$ is such that $\text{Tr}_{p^n/p}(\delta)\ne 0$, then $X+(X^p-X+\delta)^{-1}$ permutes $\f_{p^n}$. These are precisely the PRs in Theorems~\ref{T3.3} and \ref{T3.9}.
\end{rmk}

Theorems~\ref{T3.8} and \ref{T3.9} do not cover PRs of the form \eqref{3.12} for the following small values of $q$: $2,2^2,3,3^2$. For these values of $q$, PRs $f$ of the form \eqref{3.12} are determined by computer search.

\begin{itemize}
\item $q=2$. $f$ does not exist.

\item $q=2^2$. $f$ does not exist.

\item $q=3$. $f$ is equivalent to precisely one of the following:
\[
X-\frac 1{X^3-X+1},\quad X+\frac 1{X^3-X+1}.
\]

\item $q=3^2$. $f$ is equivalent to
\[
X+\frac 1{X^3-X+1}.
\]

\end{itemize}




\begin{thebibliography}{99}


\bibitem{Carlitz-DMJ-1935}
L. Carlitz, {\it On certain functions connected with polynomials in a Galois field}, Duke Math. J. {\bf 1} (1935), 137 -- 168.

\bibitem{Dickson-AM-1896}
L. E. Dickson,  {\it The analytic representation of substitutions on a power of a prime number of letters with a discussion of the linear group}, Ann. of Math.  {\bf 11} (1896-1897), 65 -- 120.

\bibitem{Fan-arXiv1812.02080}
X. Fan, {\it A classification of permutation polynomials of degree 7 over finite fields}, arXiv:1812.02080.

\bibitem{Fan-arXiv1903.10309}
X. Fan, {\it Permutation polynomials of degree 8 over finite fields of characteristic 2}, arXiv:1903.10309.

\bibitem{Ferraguti-Micheli-arXiv1805.03097}
A. Ferraguti and G. Micheli, {\it Full Classification of permutation rational functions and complete rational functions of degree three over finite fields},
arXiv:1805.03097.

\bibitem{Hou-ams-gsm-2018}
X. Hou, {\it Lectures on Finite Fields}, Graduate Studies in Mathematics 190, American Mathematical Society, Providence, RI, 2018.

\bibitem{Hou-ppt}
X. Hou, {\it Rational functions of degree four that permute the projective line over a finite field}, in preparation.


\bibitem{Hicks-Hou-Mullen-AMM-2012}
K. Hicks, X. Hou, G. L. Mullen, {\it Sums of reciprocals of polynomials over finite fields}, Amer. Math. Monthly {\bf 119} (2012), 313 -- 317.

\bibitem{Hou-Iezzi-arXiv1906.09487}
X. Hou and A. Iezzi, {\it An application of the Hasse-Weil bound to rational functions over finite Fields}, Acta Arith., to appear. 

\bibitem{Lidl-Niederreiter-FF-1997}
R. Lidl and H. Niederreiter, {\it Finite fields},  Cambridge University Press, Cambridge, 1997.

\bibitem{Li-Chandler-Xiang-FFA-2010}
J. Li, D. B. Chandler, Q. Xiang, {\it Permutation polynomials of degree $6$ or $7$ over finite fields of characteristic $2$}, Finite Fields Appl. {\bf 16} (2010), 406 -- 419.

\bibitem{Shallue-Wanless-FFA-2013}
C. J. Shallue and I. M. Wanless, {\it Permutation polynomials and orthomorphism polynomials of degree six}, Finite Fields Appl. {\bf 20} (2013), 94 -- 92.


\bibitem{Thakur-FFA-2009}
D. S. Thakur, {\it Power sums with applications to multizeta and zeta zero distribution for $\Bbb F_q[t]$}, Finite Fields Appl. {\bf 15} (2009), 534 -- 552. 

\bibitem{Williams-JNT-1975}
K. S. Williams, {\it Note on cubics over $\text{GF}(2^n)$ and $\text{GF}(3^n)$}, J. Number Theory {\bf 7} (1975), 361 -- 365.

\bibitem{Yuan-Ding-Wang-Pieprzyk-FFA-2008}
J. Yuan, C. Ding, H. Wang, J. Pieprzyk,
{\it Permutation polynomials of the form $(x^p-x+\delta)^s+L(x)$},  Finite Fields Appl. {\bf 14} (2008), 482 -- 493.


\end{thebibliography}
\end{document}